\documentclass{amsart}
\usepackage{amsmath,amsthm, amscd}
\usepackage{amsfonts}
\usepackage{txfonts}
\usepackage{marvosym}
\usepackage{graphicx}
\usepackage{color}
\usepackage{setspace}
\theoremstyle{plain}

\newtheorem{thm}{Theorem}

\theoremstyle{definition}

\newtheorem*{conj}{Conjecture}

\newtheorem*{Rem}{Remark}

\newtheorem{lem}{Lemma}

\newtheorem{cor}{Corollary}

\newcommand{\comment}[1]{}

\DeclareMathAlphabet      {\mathbfit}{OML}{cmm}{b}{it}
\newcommand{\R}{\ensuremath{\mathbb{R}}}

\begin{document}
\vspace*{-0.9cm}
\title{There exist no minimally knotted planar spatial graphs on the torus}
\author{Senja Barthel}
\address{Department of Mathematics, Imperial College London, London, SW7 2AZ, United Kingdom}
\email{s.barthel11@imperial.ac.uk}

\begin{abstract}
We show that all nontrivial embeddings of planar graphs on the torus contain a nontrivial knot or a nonsplit link. This is equivalent to showing that no minimally knotted planar spatial graphs on the torus exist that contain neither a nontrivial knot nor a nonsplit link all of whose components are unknots.
\end{abstract}

\maketitle

\section{Introduction}
\label{intro}
All considered graphs are undirected finite graphs and we will work in the piecewise linear category. A \textbf{graph embedding} is an embedding~$f:G\rightarrow  S^3$ of a graph~$G$ in $ S^3$ up to ambient isotopy and the corresponding \textbf{spatial graph}~$\mathcal{G}$ is the image of this embedding. A graph~$G$ is \textbf{planar} if there exists an embedding $f: G \rightarrow  S^2$. An embedding $f:G \rightarrow S^3$ is \textbf{trivial} if $\mathcal{G}$ is contained in a 2-sphere embedded in $S^3$. Its image~$\mathcal{G}$ is a \textbf{trivial spatial graph}. A spatial graph~$\mathcal{G}$ is \textbf{minimally knotted} if $\mathcal{G}$ is nontrivial but $\mathcal{G} - e$ is trivial for every edge $e$. Some authors call minimally knotted spatial graphs \textbf{almost trivial}, \textbf{almost unknotted} or \textbf{Brunnian}. In this paper, a nontrivial \textbf{link} is a nonsplit link with at least  two components.\\
Previous research on minimally knotted spatial graphs has been undertaken: The first example of a minimally knotted spatial graph was an embedding of a handcuff graph given by Suzuki~\cite{Suzuki1}. Kawauchi~\cite{Kawauchi4}, Wu~\cite{Wu2} and Inaba and Soma~\cite{InabaSoma} showed that every planar graph has a minimally knotted embedding. Ozawa and Tsutsumi~\cite{OzawaTsutsumi} proved that minimally knotted embeddings of planar graphs are totally knotted. Especially minimally knotted $\theta _n$-graphs have generated some interest. Kinoshita~\cite{Kinoshita} gave the first example of a minimally knotted $\theta _{3}$-graph (see Fig.~1) which Suzuki~\cite{Suzuki2} generalised to give examples of minimally knotted $\theta _{n}$-graphs for all $n \geq 3$. Closely related are \textbf{ravels} which are nontrivial embeddings of $\theta _{n}$-graphs that contain no nontrivially knotted subgraph; this definition is equivalent to the one given by Farkas, Flapan and Sullivan~\cite{FFS}. The concept of ravels has been introduced by Castle, Evans and Hyde~\cite{Ravels} as local entanglements that are not caused by knots or links and may lead to new topological structures in coordination polymers.
A ravel in a molecule has been synthesized by Lindoy~\textit{et al}~\cite{Lindoy}.
Castle, Evans and Hyde~\cite{Hyde} conjectured the following:

\begin{conj}[Castle, Evans, Hyde~\cite{Hyde}]
All nontrivial embeddings of planar graphs on the torus include a nontrivial  knot or a nonsplit link.
\end{conj}
With Theorem~\ref{goal} we prove that their conjecture is true. With \textbf{torus} we refer to an embedded torus in the \mbox{3-sphere}~$S^3$ which may be nonstandardly embedded. A \textbf{standardly embedded torus} is a torus that bounds two solid tori in~$S^3$. A nonstandardly embedded torus still bounds a solid torus in $S^3$ by the Solid Torus Theorem~\cite{Alexander}. 
\begin{thm}[Knots and links existence] \label{goal} 
Let $G$ be a planar graph and $f:G\rightarrow  S^3$ be an embedding of $G$ with image~$\mathcal{G}$. \phantom{}
If $\mathcal{G}$ is contained in the  torus~$T^2$ and contains neither a nontrivial knot nor a nonsplit link, \\then $f$ is trivial.
\end{thm}
Since~$\theta _{n}$-graphs are planar, it follows from Theorem~\ref{goal} that on the torus there exist no minimally knotted embeddings of~$\theta _{n}$-graphs with $n>2$. This gives us the following:
\vskip 10pt
\noindent
\begin{cor}[Ravels do not embed on the torus]\label{Cor}
Every nontrivial embedding of $\theta_{n}$-graphs on the torus contains a nontrivial knot.
\end{cor}
We conclude by showing that all assumptions of Theorem~\ref{goal} are necessary. Explicit ambient isotopies that transform spatial graphs that fulfil the assumptions of Theorem~\ref{goal} into the plane~$\R^2$, are given in~\cite{explicit}. Another consequence  of Theorem~\ref{goal} that is stated in the remark has been shown in~\cite{Hyde} together with~\cite{chirality}:  Nontrivial \mbox{3-connected} and simple planar spatial graphs that are embedded on a torus are chiral. A graph is \textbf{simple} if it contains no loops and no multi-edges. It is \textbf{3-connected} if at least three vertices and their incident edges have to be deleted to decompose the graph or to reduce it to a single vertex. A spatial graph is \textbf{chiral} if it is not ambient isotopic to its mirror image.

\section{Proof of Theorem~\ref{goal}}
%\noindent
%\textbf{Theorem~\ref{goal}.} (\textit{Knots and links existence}) \\
%\textit{Let $G$ be a planar graph and $f:G\rightarrow  S^3$ be an embedding of $G$ with image~$\mathcal{G}$. \\
%If $\mathcal{G}$ is contained in the  torus~$T^2$ and contains no nontrivial knot nor a nonsplit link, then $f$ is trivial.}

\subsection{Outline of the proof}
The proof uses two theorems of Scharlemann, Thompson~\cite{ScharlemannThompson} and Ozawa, Tsutsumi~\cite{OzawaTsutsumi}. We assume that the spatial graph~$\mathcal{G}$ we consider is given by an embedding $f:G \rightarrow T^2$ of a planar graph~$G$ and furthermore that $\mathcal{G}$ contains no nontrivially knotted or linked subgraph. We conclude that $\mathcal{G}$ must be trivial. During the proof, we need the following two definitions:
\vskip 10pt
\noindent
\textbf{Definition~1.}
\textit{An embedding $f:G\rightarrow  S^3$ of a graph $G$ is \textbf{primitive}, if for each component $G_i$ of $G$ and any spanning tree $T_i$ of $G_i$, the bouquet graph $f(G_{i})/f(T_{i})$ obtained from $f(G_i)$ by contracting all edges of $f(T_i)$ in $ S^3$ is trivial.}
\vskip 10pt
\noindent
\textbf{Definition~2.}
\textit{An embedding $f:G\rightarrow  S^3$ of a graph $G$ is \textbf{free}, if the fundamental group of $ S^{3}-f(G)$ is free.}
\vskip 10pt
\noindent
The argument of the proof is as follows: We start showing that the statement is true for nonstandardly embedded tori in Lemma~\ref{nonStandard}. With Lemma~\ref{conn} we argue that it is sufficient to consider connected graphs. Then we show in Lemma~\ref{bouquet} that a bouquet graph on $T^2$ either contains a nontrivial knot or is trivial. Since any connected spatial graph~$\mathcal{G}$ on $T^2$ contracts to a bouquet graph on $T^2$, it follows that $\mathcal{G}$ is primitive if it contains no nontrivial knot. By Theorem~\ref{free} we know that the restriction $f|_{G'}$ is free for all connected subgraphs $G'$ of $G$. Applying Lemma~\ref{conn} to the subgraphs $G''$ of $G$ that are not connected, we see that $f|_{G_s}$ is free for all subgraphs $G_s$ of $G$. Using Theorem~\ref{plane} we conclude that $\mathcal{G}$ is trivial.

\subsection{Preparations for the proof}
\begin{lem}[Nonstandardly embedded torus] \label{nonStandard}
Let $\mathfrak{T}^2$ be a torus that is not standardly embedded. Any spatial graph~$\mathcal{G}$ that is embedded in $\mathfrak{T}^2$ and that contains no nontrivial knot is trivial.
\end{lem}

\begin{proof}
If the spatial graph~$\mathcal{G}$ contains a cycle that follows a longitude of the torus~$\mathfrak{T}^2$, this cycle is knotted since $\mathfrak{T}^2$ itself is knotted. Therefore, no such subgraph of~$\mathcal{G}$ can exist and we find a meridian~$m$ of~$\mathfrak{T}^2$ that has no intersection with~$\mathcal{G}$. This shows that~$\mathcal{G}$ in embedded in the twice punctured sphere $\mathfrak{T}^{2}-m \simeq  S^2 - \{p_{1}, p_{2}\}$. Therefore, $\mathcal{G}$ is trivial.
\end{proof}

It follows from Lemma~\ref{nonStandard} that the statement of Theorem~\ref{goal} is true for nonstandardly embedded tori. Therefore, we only consider the standardly embedded torus~$T^2$ from now on which saves us from considering different cases.

\begin{lem}[Connectivity Lemma] \label{conn}
The image $\mathcal{G}$ of an embedding $f:G\rightarrow T^{2}\subset  S^3$ of a graph $G$ with $n >1$ connected components on the standard torus $T^2$ contains either a nonsplit link, or contains no nonsplit link and decomposes into $n$ disjoint components of which at least $n-1$ components are trivial.
\end{lem}

\begin{proof}
Take any connected component~$f(G_{i})$ of the embedding~$f(G)$ on the torus~$T^2$. The complement of $f(G_{i})$ \textit{in the torus} (without considering the rest of the spatial graph~$f(G - G_i)$) is a collection of pieces that can be the punctured torus, discs, and essential annuli without boundaries. (An essential annulus contains a simple closed curve that does not bound a disc in the torus.) \\
\indent
In the case that the complement of $f(G_i)$ in $T^2$ includes the punctured torus, $f(G_i)$ is trivial and splits from the other components.\\
\indent
If the complement of $f(G_i)$ in $T^2$ is only a collection of discs, then all other components of $f(G)$ lie in one of those discs and therefore are trivial and the graph is split. ($f(G_i)$ might or might not contain a nonsplit link.)\\
\indent
In the case that the complement of $f(G_i)$ in $T^2$ includes an essential annulus~$A$, it is possible that other components of $G$ are embedded in this annulus. A component $G_j$ might be embedded in the annulus in two ways: Either the complement of $f(G_j)$ in $A$ includes a punctured annulus and therefore $f(G_j)$ is trivial and splits from the rest of the spatial graph $f(G-G_j)$. Or $A - f(G_j)$ contains two annuli. The annulus $A$ has one type of an essential curve $c$ running inside it; $c$ is parallel to the boundary curves of $A$.
In the case that $A - f(G_j)$ contains two annuli, a subgraph of $f(G_j)$ must be deformable to be parallel to $c$. If $c$ is a meridian or a prefered longitude of $T^2$, both components $f(G_i)$ and $f(G_j)$ are split and trivial since the torus is a standard torus. If $c$ is neither a meridian nor a longitude of $T^2$, $f(G_i)$ and $f(G_j)$ are nonsplittably linked.
\end{proof}

\begin{lem}[Bouquet Lemma] \label{bouquet}
The image $\mathcal{B}$ of an embedding $f:B\rightarrow T^{2}\subset  S^3$ of a connected bouquet graph $B$ on the torus $T^2$ either contains a nontrivial knot or is trivial. 
\end{lem}

\begin{proof}
A bouquet graph~$\mathcal{B}$ on~$T^2$ that contains no nontrivial knot contains only cycles which all are the unknot by assumption. The unknot on the torus can take the following forms: 
\begin{enumerate}
\item $T(0,0)$ loop that bound a disc in $T^2$ (trivial elements in $\pi_{1}(T^2)$), 
\item $T(0,1)$ meridional loop,
\item $T(1,0)$ longitudinal loop,
\item $T(1,n)$ loop or alternatively $T(n,1)$ loop, $n \geq 1$
\end{enumerate}
\noindent
Loops of type (1) do not contribute to the nontriviality of~$\mathcal{B}$.\\
%Loops of type (1) are trivial by definition.
If~$\mathcal{B}$ has loops of the types~(1),~(2)~and~(3) only, it is trivial.\\
If~$\mathcal{B}$ has loops of type~(4), there are -- beside the loops $T(0,0)$ -- only three types of loops simultaneously embeddable on the torus without self-intersections: $T(0,1), T(1,n)$ and $T(1,n+1)$ (respectively $T(1,0), T(n,1)$ and $T(n+1,1)$). This can easily be confirmed by applying the formula of Rolfsen's exercise~2.7~\cite{Rolfsen}: If two torus knots $T(p,q)$ and $T(p', q')$ intersect in one point transversally, then $p q' - q p' = \pm 1$. Such a bouquet is trivial.
\end{proof}

\begin{thm}[Ozawa and Tsutsumi's freeness criterion~\cite{OzawaTsutsumi}] \label{free} 
An embedding $f:G\rightarrow S^3$ of a graph $G$ is primitive if and only if 
the restriction $f|_{G'}$ is free for all connected subgraphs $G'$ of $G$.
\end{thm}

\begin{thm}[Scharlemann and Thompson's planarity criterion~\cite{ScharlemannThompson}] \label{plane} 
An embedding $f:G\rightarrow S^3$ of a graph $G$ is trivial if and only if\\
(a) $G$ is planar and \\
(b) for every subgraph $G_s \subset G$, the restriction $f|_{G_s}$  is free.
\end{thm}
\vskip 0.3cm
\subsection{The proof}

We are now ready to prove Theorem~\ref{goal} and Corollary~\ref{Cor}:

\begin{proof} (of Theorem~\ref{goal}).
It follows from Lemma~\ref{nonStandard} that the statement of Theorem~\ref{goal} is true for nonstandardly embedded tori. Therefore, we assume that $\mathcal{G}$ is embedded in the standard torus~$T^2$. Since $\mathcal{G}$ contains no nonsplit link by assumption, we can assume by Lemma~\ref{conn} that $G$ is connected. Any connected spatial graph contracts to a spatial bouquet graph~$\mathcal{B}$ if a spanning tree~$T$ is contracted in~$ S^3$. If the spatial graph is embedded in a surface, edge contractions can be realised in the surface. It follows that contracting a spanning tree of a connected spatial graph that is embedded in~$T^2$ results in a bouquet graph that is embedded in $T^2$ itself. Since $\mathcal{G}$ contains no nontrivial knot by assumption, $\mathcal{B}$ also contains no nontrivial knot. We know from Lemma~\ref{bouquet} that a bouquet graph that is embedded in the torus~$T^2$ and that contains no nontrivial knot is trivial. Therefore it follows that, for any chosen spanning tree $T$ of $G$, the bouquet graph~$\mathcal{B}=f(G)/f(T)$ which is obtained from $f(G)$ by contracting all edges of $f(T)$ in $S^3$ is trivial. Consequently $f$ is primitive by definition. By Theorem~\ref{free}, the restriction $f|_{G'}$ is free for all connected subgraphs $G'$ of $G$. Let $G''$ be a subgraph of $G$ that is not connected. Since $G''$ is a subgraph of $G$,   it does neither contain nontrivial links nor nontrivial knots by assumption. Applying Lemma~\ref{conn} to $G''$ shows that the connected components of $f|_{G''}$ are split and at most one connected component $f|_{G''_1}$ of $f|_{G''}$ is not trivial. Therefore, the restriction $f|_{G''}$ is free if and only if $f|_{G''_1}$ is free. Since $G''_1$ is a connected subgraph of $G$, we know already that $f|_{G''_1}$ is free. Therefore, the restriction $f|_{G_s}$ is free for all subgraphs $G_s$ of $G$. As~$G$ is planar by assumption, it follows from Theorem~\ref{plane} that $f$ is trivial. 
\end{proof}

\begin{proof}(of Corollary~\ref{Cor}).
As there exists no pair of disjoint cycles in a $\theta_{n}$-graph, such a graph does not contain a nontrivial link. Since $\theta_{n}$-graphs are planar, the statement of the corollary follows directly from Theorem~\ref{goal}.
\end{proof}
It has been shown in~\cite{Hyde} together with~\cite{chirality} that every nontrivial embedding of a simple 3-connected spatial graph on the torus that contains a nontrivial knot or a nonsplit link is chiral. The following remark is therefore a consequence of Theorem~\ref{goal}.
\begin{Rem}[Chirality]
Nontrivial embeddings of simple 3-connected planar graphs in the torus are chiral.
\end{Rem}

%%%%%%%%%%%%%%%%%%%%%%%%%%%%%%%%%%%%%%%%%%%%%%%%
\vskip 0.3cm
\subsection{All assumptions that have been made are necessary.}\hfill\\
This can be seen by considering the following examples:
\begin{itemize}
\item 
There exist nontrivial embeddedings on $T^2$ that contain neither a nontrivial knot nor a nonsplit link.\\
These are embeddings of nonplanar graphs.\\
\textit{Examples}: $K_{3,3}$ and $K_5$ embedded as shown left in Fig.~1.
\item
There exist nontrivial embeddings of planar graphs that contain neither a nontrivial knot nor a nonsplit link.\\
These embeddings are not embedded on the torus.\\
\textit{Examples}: Kinoshita-theta curve (middle in Fig.~1) and every ravel.
\item
There exist nontrivial embeddings of planar graphs on $T^2$.\\
\textit{Examples}: Spatial graphs that are subdivisions of nontrivial torus knots with $n>0$ vertices and $n$ edges (right in Fig.~1).
\end{itemize}
\begin{figure}[h]
	\centering
	\def\svgwidth{350pt}
	 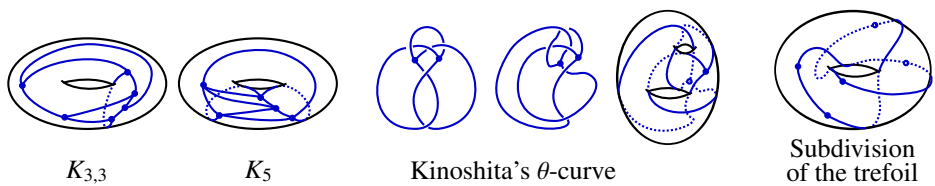
\vspace*{-8pt}
\caption{All assumptions are necessary.\label{1}}
\end{figure}

\vskip 0.6cm
\section*{Acknowledgments}
I thank Tom Coates, Erica Flapan, Youngsik Huh, Stephen Hyde, Danielle O'Donnol, Makoto Ozawa, Matt Rathbun and Kouki Taniyama for helpful comments and discussions. I also want to thank my PhD supervisor Dorothy Buck under whose supervision the research was undertaken. It was financially supported by the Roth studentship of Imperial College London mathematics department, the DAAD, the Evangelisches Studienwerk, the Doris Chen award, and by a JSPS grant awarded to Kouki Taniyama.

\end{document}